\documentclass[onecolumn,notitlepage,showpacs,11pt]{revtex4-2}
\usepackage{xcolor}
\usepackage[utf8]{inputenc}
\usepackage{graphicx}
\usepackage{tikz-cd}
\usepackage{dcolumn}
\usepackage{amsmath}
\usepackage{amssymb}
\usepackage{mathtools}
\usepackage{cancel}
\usepackage{hyperref}
\usepackage{float}
\allowdisplaybreaks

\newtheorem{theorem}{Theorem}[section]
\newtheorem{corollary}[theorem]{Corollary}
\newtheorem{proposition}[theorem]{Proposition}

\newtheorem{definition}[theorem]{Definition}
\newenvironment{proof}[1][Proof]{\begin{trivlist}
\item[\hskip \labelsep {\bfseries #1}]}{\end{trivlist}}

\newenvironment{remark}[1][Remark]{\begin{trivlist}
\item[\hskip \labelsep {\bfseries #1}]}{\end{trivlist}}

\newcommand{\qed}{\nobreak \ifvmode \relax \else
	\ifdim\lastskip<1.5em \hskip- \lastskip
	\hskip 0.5em plus0em minus0.5em \fi \nobreak
	\vrule height0.75em width0.5em depth0.25em\fi}

\begin{document}

\title{Conformal Killing--Yano Ricci solitons: Basic structure and first examples}

\author{Mohammadjavad \surname{Habibivostakolaei}}
\email{mjhabibi@hnas.ac.cn}
\affiliation{Institute of Mathematics, Henan Academy of Sciences, 228 Mingli Road, Zhengzhou 450046, Henan, China}

\author{Yen-Kheng \surname{Lim}}
\email{yenkheng.lim@gmail.com, yenkheng.lim@xmu.edu.my}
\affiliation{School of Mathematics and Physics, Xiamen University Malaysia, 43900 Sepang, Malaysia}

\author{Abbas M \surname{Sherif}}
\email{abbasmsherif25@gmail.com}
\affiliation{Institute of Mathematics, Henan Academy of Sciences, 228 Mingli Road, Zhengzhou 450046, Henan, China}

\begin{abstract}
We introduce a geometric structure -- a conformal Killing--Yano Ricci soliton (CKY--RS) -- that couples conformal Ricci soliton (CRS) geometry to conformal Killing--Yano (CKY) 2--forms. The soliton field of the CRS geometry is given by the divergence of the CKY 2--form $\xi^a$. It is also observed that, in 4--dimensional Lorentzian signature, the CKY--RS structure does not preserve the principal CKY structure under a non-trivial conformal deformation, and therefore, the principal conformal CKY 2-form associated to any non-trivial CKY--RS structure on a local Kerr--NUT--(A)dS background or a background conformal to a local Kerr--NUT--(A)dS is not closed. For Einstein backgrounds of arbitrary dimension and signature, the CKY integrability forces $\xi^a$ to be Killing vector for the background metric $g_{ab}$ and a conformal Killing vector for the soliton (conformal) metric $\tilde g_{ab}$, and the soliton (i.e. conformal) metric is necessarily Einstein, with $\tilde R=c \tilde g_{ab}$ and the soliton constant given by $\lambda=c-\alpha$. The conformal factor satisfies an eigenvalue equation and an Obata--type Hessian equation, yielding the classical Brinkmann–Tashiro–Obata–K\"{u}hnel type problem, where the CKY--RS coupling specifies the constant $c$ as the Einstein constant for the conformal metric, expressed through the constants of the first integral of the conformal factor, where one of these constants determines the Ricci curvature of the fiber of the split product on the background. If $\alpha=0$, $\xi^a$ is a KV for both $g_{ab}$ and $\tilde g_{ab}$, and if $\alpha\neq0$, $\tilde g_{ab}$ is Ricci-flat. Locally, $g_{ab}$ is a warped product, and include products, Ricci--flat cones, and metrics of the Brinkmann class, with solutions controlled by the CKY--RS data. For $\alpha=0$, the CKY--RS data reduces to a Killing data on the fiber. If $\alpha\neq0$, $\xi^a$ is a proper homothety of the fiber. In the compact Riemannian case (or $\xi^a$ has closed orbits), $\alpha=0$ and $(M,g_{ab})$ is isomorphic to the round sphere. Finally, we give explicit constructions for static spherically symmetric geometries and BTZ backgrounds, including a CKY--RS realization with a time--dependent locally flat conformal representative. These results provide a geometric framework for studying CRS with hidden symmetry structure, with potential applications to exact geometries in general relativity.  
\end{abstract}

\maketitle

\section{Introduction}


Ricci solitons (RS), which are self--similar solutions to the Ricci flow, have connected geometry, geometric analysis, and gravitational physics, and served as a model for singularity formation and gravitational collapse \cite{Hamilton1,Hamilton2,Perelman1,Perelman2}. These structures appear as gravitational Instantons and in the studies of holographic dualities, etc. 

A natural generalization of Ricci solitons is conformal Ricci soliton (CRS), where the metric is allowed to change under a diffeomorphism $g\to \tilde g:=e^{2u}g$, Catino \textit{et al.} \cite{Catino1}. These structures include the Einstein metrics, conformally Einstein metrics, and Ricci solitons as special cases \cite{Brinkmann1,Cao1,Cao2,Gover1}. These structures have also been investigated on contact manifolds and in connection to almost Ricci solitons. Rigidity results, mostly under compactness, constant scalar curvature, or gradient assumptions, have been established (see \textit{et al.}, \cite{Catino1}). 

On the other hand, Conformal Killing--Yano (CKY) tensors which generalize Killing–Yano tensors, go back to works by Tachibana \cite{Kashiwada1,Kashiwada2} provides and is another source of rigidity. Yano was the first to identify the connection between Killing forms and the existence of first integrals of the geodesic equations on manifolds \cite{Yano1}. The first direct application to integrating equations of motion was implemented by Walker and Penrose in \cite{Walker1}. CKY tensors now feature extensively in mathematical relativity and supergravity theories. The closed CKY forms and associated vector fields encode hidden symmetry structures and play a fundamental role in the geometry of Kerr–NUT–(A)dS spacetimes and related integrability properties \cite{Kashiwada2,Frolov1,Frolov2,Krtous1}. Since a CKY 2--form is conformally covariant, with the Kerr--NUT--(A)dS classification, one expects there to exist a family of metrics conformal to the Kerr--NUT--(A)dS family admitting CKY 2--form and such a family would deserve classification in their own right. For a closed CKY tensor, in particular, the symmetric derivative of the associated vector field is controlled by the Ricci tensor. As a consequence, on an Einstein manifold the associated vector becomes a Killing vector field. This simple observation suggests that the CKY tensor and a RS vector field need not be regarded as independent pieces of geometric data: the hidden symmetry itself can provide the vector field for the soliton structure. The relation between CCKY tensors and the Kodama vector field in warped product geometries provides another example of the geometric significance of this structure and emphasizes the geometric significance of this vector field in spherical symmetry \cite{Kinoshita1}.

These are all strong motivations to connect these two structures: CKY tensors and conformal Ricci solitons. The purpose of this paper is to systematize this connection. We introduce a \textit{conformal Killing--Yano Ricci soliton} (CKY--RS) as a CRS in which the soliton vector $\xi^a$ is precisely the vector field associated with a CKY 2--form through its divergence. Hence the conformal factor, the CKY tensor, and the soliton vector are coupled. The compatibility of the coupling through the CKY--RS produces additional curvature constraints that are absent when the two structures are considered independently. The soliton vector is the CKY vector of the background, not of the conformal metric. This choice ties the soliton equation to the integrability of the background CKY structure.

In 4--dimensional Lorentzian spacetimes, assuming that the CKY 2--form is the principal CKY 2-form, i.e. closed and non-degenerate, the conformal representative is locally Kerr--NUT--(A)dS by \cite{Krtous1}, provided the principal CKY structure is preserved. This forces the conformal factor to be constant. Thus, the principal hidden symmetry structure in this setting obstructs the existence of non-trivial CKY--RS structures under the conformal map. In other words, a non-trivial CKY--RS structure does not preserve the principal CKY structure.

For Einstein backgrounds, the CKY--RS equations reduce to an eigenvalue equation for the conformal factor $\xi(\phi)=\alpha\phi$, and an overdetermined Obata--type Hessian equation $\mbox{Hess}(\phi)=\Psi g$ for some function $\Psi$, with $\phi=e^{-u}$. The Hessian equation is the classical Brinkmann–Tashiro–Obata–K\"{u}hnel problem \cite{Brinkmann1,Obata1,Brinkmann2,Tashiro1,Kuhnel1} and therefore the solutions are not new for our structure. The proofs are briefly recalled to keep track of constants of the solutions that are implicit in the classical results. However, the CKY--RS coupling specifies the constant $c$ as the Einstein constant for the conformal metric, expressed through the constants of the first integral of the conformal factor, where one of these constants determines the Ricci curvature of the fiber of the split product on the background (these include products, Ricci--flat cones, and metrics of the Brinkmann class). The CKY integrability forces the soliton field to be Killing field for the background and a conformal Killing field for the conformal metric $\tilde g_{ab}$ which is necessarily Einstein. The vanishing or non-vanishing of $\alpha$ characterizes the flow of the fiber metric of the product along the soliton field, whose explicit form is governed by the CKY--RS data. And if $(M,g_{ab})$ is compact Riemannian (or $\xi^a$ has a closed orbit), $\alpha$ necessarily vanishes, and $(M,g_{ab})$ is isomorphic to the round sphere.

Finally, we demonstrate explicit CKY--RS realizations on some Lorentzian geometries. In a spherically symmetric spacetime, the CCKY 2--form is unique and its associated vector is the Kodama vector. This gives a direct realization of the proposed CKY--RS structure in a setting that is of particular relevance to general relativity, where the Kodama vector provides a ``time'' direction when there is no stationary Killing field \cite{Kinoshita1}. For static and spherically (circularly) symmetric metrics, CKY--RS structures yields a parametrized family of metrics. On the 3--dimensional BTZ background, the CKY--RS structure produces a time-dependent conformally flat representative. These examples complement the rigidity results and illustrate how CKY--RS structures can produce non--trivial conformal RS geometries. More broadly, the CKY--RS framework provides a geometric setting in which a combination of conformal soliton structures and the geometry of CKY tensors may have potential applications to exact spacetime geometries.

The paper is organized as follows. In Section \ref{sec-2}, we briefly review RS, CRS, and CKY tensors, and fix our conventions. Section \ref{sec-3} introduces the CKY--RS structure and remark on its implcation for 4--dimensional Lorentzian metrics in relation to the Kerr--NUT--(A)dS family. In Section \ref{sec-4}, we specialize to CKY--RS structures with Einstein backgrounds and develop a resulting spectral obstruction and classification from an Obata--type Hessian equation. Section \ref{sec-5} gives explicit CKY--RS constructions on spherically symmetric and BTZ geometries. In Section \ref{sec-6}, we discuss the geometric and general relativistic implications of the construction and outline several potential directions for future work.

Throughout, the paper emphasizes the geometry of the coupled system of the CKY--RS structure. The connection with general relativity is not simply motivational: The ingredients of Einstein Lorentzian metrics, principal CKY forms, Petrov type D geometries, the Kerr--NUT--(A)dS family, the Kodama vector, and BTZ geometry all arise naturally as some realizations of the CKY--RS equations. Thus, CKY--RS construction provides a framework in which hidden symmetry and geometries from the Ricci--flow geometry can be studied within a common conformal--geometric setting.


\section{Preliminaries}\label{sec-2}


In this section we introduce a new geometric structure: a conformal Killing--Yano Ricci soliton (CKY--RS). This structure connects the dynamics of the Ricci tensor via the Hamilton's normalized Ricci flow \cite{Hamilton1,Hamilton2} and integrability of geodesic and Hamilton-Jacobi equations via hidden symmetries (see, e.g. \cite{Kashiwada1,Kashiwada2,Frolov1,Frolov2}).

\subsection{Conformal Ricci soliton}

A \textit{Ricci soliton} (RS) is a metric together with a (smooth) vector field $y^a$ on a manifold $M$ which satisfies \cite{Hamilton1,Hamilton2}
\begin{align}
R_{ab}+\frac{1}{2}\mathcal{L}_yg_{ab}=\lambda g_{ab},\label{ricci-soliton}
\end{align}
where $R_{ab}$ is the Ricci tensor, $y^a$ is the \textit{soliton vector field} and $\lambda$ is the \textit{soliton constant}. The structure itself is given by the triple $(M,g_{ab},y^a,\lambda)$. The soliton is called \textit{shrinking}, \textit{steady}, or \textit{expanding} if the soliton field satisfies $\lambda>0,\lambda=0$, or $\lambda<0$, respectively. 

RS are considered generalizations of the Einstein metrics, and nontrivial RS represent self-similar deformations of Einstein geometries generated by the (non--trivial) $y^a$--flow. When $y^a$ is exact, i.e. $y_a=\nabla_af$ for some function $f$, the soliton is called a gradient RS. Gradient RS play a central role in geometric analysis and in singularity formation under the Ricci flow.

\begin{definition}
A manifold $(M,g_{ab})$ is called a \textbf{conformal Ricci soliton} (CRS) if there exists a conformal change of metric $\tilde{g}_{ab}=e^{2u}g_{ab}$, with $u\in C^{\infty}(M)$, such that the conformal metric satisfies the RS equation \eqref{ricci-soliton} \cite{Catino1}. The quintuple $(M,g_{ab},u,y^a,\lambda)$ is called the CRS structure.
\end{definition}

CRS therefore generalizes the usual Ricci soliton by imposing the soliton structure on a conformal metric. The soliton vector field on a CRS is not imposed and is in general determined by the manifold geometry. The main purpose of the CKY-RS construction is to provide a natural geometric origin of this vector field.

\subsection{Conformal Killing-Yano tensors}

We briefly introduce the conformal Killing--Yano tensor which is the second ingredient of the new geometric structure. This tensor generalizes Killing--Yano tensors and encode hidden symmetries that are responsible for the integrability and separability of the geodesic and Hamilton--Jacobi equations in four and higher dimensional gravity theories (see, for example \cite{Walker1,Frolov1,Frolov2}).
\begin{definition}
A rank--two conformal Killing Yano (CKY) tensor on an $n$--dimensional pseudo--Riemannian manifold is a 2--form $Y_{ab}$ satisfying \cite{Kashiwada1,Kashiwada2}
\begin{align}\label{cky-eq}
\nabla_cY_{ab}=2g_{c[a}\xi_{b]}+z_{abc},\quad z_{abc}=z_{[abc]}=(dY)_{abc},\quad \xi_a=\frac{1}{n-1}\nabla^bY_{ba}.
\end{align}
The equation \eqref{cky-eq} is the (rank 2) CKY equation, and the vector field $\xi_a$ obtained from the divergence of $Y_{ab}$ will be called the CKY vector field. The notation $\xi[Y]$ will sometimes be used to indicate $\xi^a$ is the vector field associated to the CKY 2--form $Y_{ab}$. Furthermore, since $\xi^a$ is the divergence of a 2--form, it is divergence--free.
\end{definition}

A CKY tensor $Y_{ab}$ is said to be closed if
\begin{align}
z_{abc}=(dY)_{abc}=0,
\end{align} 
abbreviated CCKY tensor. CCKY tensors are of particular interest in gravity and supergravity theories because their integrability conditions impose strong restrictions on the curvature tensor. 

A crucial integrability condition is a rather neat relation on the Ricci tensor:
\begin{align}
\nabla_{(a}\xi_{b)}=-\frac{1}{n-2}R_{c(a}{Y_{b)}}^c.\label{ccky-int1}
\end{align}
Therefore, when the $g_{ab}$ is Einstein, the anti-symmetry of $Y_{ab}$ forces $\nabla_{(a}\xi_{b)}=0$. That is, the CKY vector field $\xi_a$ becomes a Killing vector. This is in fact the point: hidden symmetries encoded in CCKY tensors generate spacetime symmetries on Einstein metrics. Some details of classification results and integrability can be found in \cite{Frolov1,Frolov2,Krtous1} and associated references. We emphasize that the Einstein condition is sufficient but not necessary for promoting $\xi_a$ to a KV: It is possible, in general, to have a background whose matter fields have a very special alignment with the CKY tensor for which the right hand side of \eqref{ccky-int1} vanishes. We also note that, since $\xi^a$ is the divergence of a 2-form $Y_{ab}$, its divergence is zero. Thus, $\xi^a$ is a CKV if and only if it is a KV.

CKY tensors are widely employed in analyzing integrability of curved backgrounds. A principal CKY 2--form implies Petrov type D (or O) in 4--dimensional spacetimes, and it underpins the hidden symmetry and separability structure in the Kerr--NUT--(A)dS class of metrics \cite{Walker1,Frolov1,Frolov2}. Higher-dimensional generalizations for black holes yield the principal Killing--Yano p--forms which render completely integrable the geodesic equations of motion. More recently, Kinoshita \cite{Kinoshita1} showed that the Kodama vector field of a warped-product spacetime arises as the associated vector field of a CCKY 2--form; we will return to this in Section \ref{sec-5}. On the other hand, the integrability relation \eqref{ccky-int1} shows that a CCKY 2--form is much more than providing a hidden symmetry: it constrains the Ricci tensor through its associated vector field. In particular, on an Einstein background CKY vector field is a KV, so that the hidden symmetry encoded by the CCKY tensor generates an ordinary spacetime symmetry. This observation will be important henceforth, where we identify this CKY--aligned vector field with the soliton vector of a conformal Ricci soliton.

Thus, the CKY tensor and the conformal Ricci soliton are not independent structures in the construction considered here. The tensor canonically fixes the soliton vector, while its integrability conditions constrain the geometry on which the soliton is defined.

Unless explicitly stated otherwise, all classification and rigidity statements are local. For example ''a conformal representative is Kerr--NUT--(A)dS'' means locally isometric to a member of that family on a suitable open set. Global results require additional assumptions concerning completeness, topology, connectedness, regularity of the conformal factor, and the global behavior of the CKY tensor, which is beyond the scope of this paper. The explicit BTZ CKY--RS construction for example  of Section \ref{sec-5} illustrates this transparently: the conformal factor vanishes at at the BTZ horizon, and the coformal factor diverges there. Therefore the corresponding CKY--RS structure is defined only on the region outside the BTZ horizon.

\section{CKY--RS structure}\label{sec-3}

In this section we introduce conformal Killing--Yano Ricci solitons and derive some integrability constraints.

\begin{definition}\label{cky-rs-structure}
A sextuple $(M,g_{ab},u,Y_{ab},\xi^a[Y],\lambda)$ is called a \textbf{CKY Ricci soliton} (CKY--RS or CKY--RS structure) if the CRS structure $(M,g_{ab},u,y^a,\lambda)$ is sourced by the CKY vector field $y^a=\xi^a[Y]$ of the CKY 2--form. That is, the underlying vector field is geometrically fixed by $Y_{ab}$. The CKY--RS will be called trivial if $u$ is constant, i.e. the conformal transformation is homothetic.
\end{definition}

The CKY--RS structure refines the CRS structure by tying its soliton field to the hidden symmetries encoded in a CKY tensor. Henceforth, we will alternative write $\xi^a$ in place of the soliton vector $\xi^a[Y]$ associated to the CKY--RS structure when there is no ambiguity.

A particular simplification occurs when the CKY vector field is a conformal Killing vector (CKV) for the background metric $g_{ab}$: $\mathcal{L}_{\xi}g_{ab}=2\sigma g_{ab}$, for some function $\sigma$. Then, the soliton equation simplifies to $\tilde R_{ab}=c\tilde g_{ab}$, where 
\begin{align}
\lambda-(\mathcal{L}_{\xi}u+\sigma)=c.\label{funda7}
\end{align}
By Schur's lemma, $c$ is constant and the conformal metric is Einstein. That is, the CKY--RS structure with a soliton vector field that is a CKV ``selects'' a representative Einstein metric in the conformal class. Importantly, the conformal factor is not unconstrained but requires a particular alignment with $\xi^a$ via \eqref{funda7}.

\begin{remark}[Remark 2.]\label{rem-cky-vec}
In the Definition \ref{cky-rs-structure} the soliton vector is $\xi^a$, the divergence of $Y_{ab}$ computed with respect to the background metric $g_{ab}$, while the soliton equation is imposed on the conformal metric $\tilde g_{ab}=e^{2u}g_{ab}$. This is deliberately done to tie the integrability of the background to the soliton geometry. The background datum $(g_{ab},Y_{ab})$ fixes $\xi^a$. The conformal factor remains the only unknown and the soliton equation becomes a constraint on the conformal class. If we are to instead us the alternative $\tilde\xi=\tilde\xi_{a(\tilde g)}=e^u(\xi_a+u^bY_{ab})$, this depends on the factor $u$ and hence not conformally covariant. The soliton equation on the other hand would involve only the conformal data with $u$ being a pure gauge and is therefore independent of the choice of representative $g_{ab}$ in the conformal class, i.e. $u$ would carry no information related to the structure.
\end{remark}

\begin{remark}[Remarks in dimension 4.] Consider dimension 4 with Lorentzian signature. Let $Y_{ab}$ be a CCKY tensor and non-degenerate ($Y\wedge Y\neq0$), i.e. a principal CKY 2-form. By the classification of Krtous \textit{et al.} \cite{Krtous1}, $\tilde g_{ab}$ is locally Kerr--NUT--(A)dS provided that $\tilde Y_{ab}=e^{3u}Y_{ab}$ is closed. Since $d\tilde Y=3e^{3u}du\wedge Y$, and the map $\beta\mapsto\beta\wedge Y$, for a 1-form $\beta$, is an isomorphism in dimension 4, $d\tilde Y_{ab}=0$ forces $du=0$ and $u$ is constant. Thus, A CKY--RS structure cannot preserve the principal CKY structure under a non-trivial conformal deformation. In general $\tilde Y_{ab}$ is not closed
\end{remark}

We next turn to the Einstein-background case, without dimennsional fixing, where the CKY--RS equations are further simplified and lead to stronger classification results.


\section{Einstein background rigidity and classification}\label{sec-4}


We now specialize to the case in which the background metric $g$ is Einstein. In this setting, the CKY integrability condition forces the associated vector field $\xi[Y]$ to be Killing, and the CKY--RS equations reduce to a tractable system for the conformal factor. All manifolds are assumed connected and have dimensions $n\geq3$.

\subsection{CKY--RS existence and spectral obstructions}

Let $g_{ab}$ be Einstein. Then, $R_{ab}=\Lambda g_{ab}$ for some constant $\Lambda$. Therefore, $\nabla_{(a}\xi_{b)}=0$, and $\xi$ is a KV. Hence, $\tilde R_{ab}=c\tilde g_{ab}$, with $c=\lambda-\xi(u)$. This is our CKY--RS compatibility relation \eqref{funda7} and its centrality to our structure will now be exploited.
\begin{definition}\label{cky-rs-flow}
Let $(M,g_{ab})$ be an Einstein manifold and $Y_{ab}$ a CKY 2--form on $(M,g_{ab})$ with associated vector field $\xi$. Define the set
\begin{align*}
\Sigma_{\xi}:=\{\alpha\in\mathbb{R}:\exists \psi>0, \rm d\psi\not\equiv0,\mathcal{L}_\xi\psi=\alpha\psi\}.
\end{align*}
Then any CKY--RS structure compatible with $Y_{ab}$ produces a pair $(\alpha,\psi)\in \Sigma_{\xi}$, with $\alpha=c-\lambda$.
\end{definition}

From Definition \ref{cky-rs-flow}, if $\Sigma_{\xi}$ is empty, no CKY--RS structure compatible with $Y_{ab}$ can exist. This immediately leads to the following obstruction result.
\begin{proposition}
Let $(M,g_{ab})$ be an Einstein manifold and $Y_{ab}$ a CKY 2--form on $(M,g_{ab})$ with associated vector field $\xi$. If a non-trivial CKY--RS structure compatible with $Y_{ab}$ exists with $\widetilde g=\phi^{-2}g$, with $\phi>0$, then there exists a constant $\alpha\in\mathbb R$ such that $\mathcal{L}_\xi\phi=\alpha\phi$. Consequently, if the Lie operator $\mathcal{L}_{\xi}$ admits no positive non-constant eigenfunction for any real eigenvalue, then $g_{ab}$ admits no non-trivial CKY--RS structure compatible with $Y_{ab}$.
\end{proposition}

Viewing CKY--RS through the dynamics of the conformal factor provides some insights into the structure of the metrics of the CKY--RS structure, using elementary arguments. Denote by $\Phi_t$ the flow generated by $\xi^a$, and fix $p=p_0\in M$. Along the curve $\gamma_p(t)=\Phi_t(p)$, the equation $\mathcal{L}_\xi\phi=\alpha\phi$ becomes
\begin{align*}
\frac{d}{dt}\phi(\gamma_{p_0}(t)) =\alpha\phi(\gamma_{p_0}(t)).
\end{align*}
This equation has the explicit solution $\psi(\Phi_t(p_0))=e^{\alpha t}\psi(p_0)$, with initial condition $\psi(\gamma_{p_0}(0))=\psi(p_0)$. Thus, $\psi$ is invariant when $\alpha=0$, grows exponentially when $\alpha>0$, and decays exponentially when $\alpha<0$.

The preceding analysis gives a necessary condition, through a CKY--RS--flow, for the existence of a non--trivial CKY--RS structure on an Einstein background. In particular, the conformal factor must be a positive (non-constant) eigenfunction of the operator $\mathcal{L}_{\xi}$. Periodic orbits, for example, forces $\alpha=0$. Hence the CKY--RS dynamics alone already provides obstruction to the existence of non-trivial CKY--RS structure for broad classes of Einstein backgrounds.

The remaining Einstein background equations impose an additional differential condition on the conformal factor. We next turn to these constraints and derive classification results.

\subsection{Einstein background geometry and classification}

We now examine the geometric constraints imposed by the Einstein background equations. Let $(M,g_{ab},u,Y_{ab},\lambda)$ be a CKY--RS structure and let
\begin{align}
\mathcal{G}=\Delta u+(n-2)|\nabla u|^2+ce^{2u},\quad u_{ab}-u_au_b=Kg_{ab},\label{hess-u}
\end{align}
for some function $K$. Using the soliton equation together with the transformation law for the Ricci tensor, the background metric for the CKY--RS structure satisfies
\begin{align}
R_{ab}=(\mathcal{G}+(n-2)K)g_{ab},\quad K=\frac{\Lambda-\mathcal{G}}{n-2}.\label{bec-2}
\end{align}

Set $\phi=e^{-u}$ and $\phi_{ab}=\nabla_a\nabla_b\phi$. Then, the Hessian equation for $u$, \eqref{hess-u} becomes
\begin{align}
\phi_{ab}=-K\phi g_{ab},\label{bec-5}
\end{align}
We refer to \eqref{bec-5} as the background Einstein condition (BEC). We recognize \eqref{bec-5} as an Obata--type Hessian equation. 

\begin{proposition}
et $(M,g_{ab},u,Y_{ab},\lambda)$ be a CKY--RS structure with $g$ Einstein. Then,
\begin{align}
K&=\frac{\Lambda}{n-1}-\frac{q}{\phi},\label{K-form}\\
(n-1)|\nabla\phi|^2&=-(\Lambda\phi^2+c)+2(n-1)q\phi,\label{first-int-1}
\end{align}
for some constant $q$.
\end{proposition}
\begin{proof}
For a Hessian equation of the form $\phi_{ab}=\psi g_{ab}$, for some function $\psi$, commuting derivatives and using the Ricci identities gives
\begin{align*}
R_{cabd}^{\ \ \ \ d}\phi_d=g_{ab}\psi_c-g_{cb}\psi_a.
\end{align*}
Tracing over $c$ and $b$ and using $R_{ab}=\Lambda g_{ab}$ gives $\psi_a=-(\Lambda/(n-1))\phi_a$, which integrates to
\begin{align}
\psi=-\frac{\Lambda}{(n-1)}\phi +q,
\end{align}
for some constant $q$. Setting $\psi=-K\phi$ gives \eqref{K-form}. Substituting \eqref{K-form} into \eqref{bec-2} and using the definition of $\mathcal{G}$ gives \eqref{first-int-1}.\hfill\qed  
\end{proof}

The relation \eqref{first-int-1} is a first integral for the BEC \eqref{bec-5}.

Suppose $d\phi=0$ (or equivalently $du=0$). Then, $K=0$. Also, $c=\lambda$ and $\mathcal{G}=\phi^{-2}\lambda$. Therefore, the background scalar curvature controls the nature of the soliton.
\begin{proposition}
Let $(M,g_{ab},u,Y_{ab},\lambda)$ be a Einstein CKY--RS structure, and let $\phi=e^{-u}$. If $d\phi=0$. Then the soliton is steady, shrinking, or expanding if $R=0, R>0$, and $R<0$, respectively.
\end{proposition}

\begin{proposition}\label{c-c}
Let $(M,g_{ab},u,Y_{ab},\lambda)$ be a Einstein CKY--RS structure, and let $\phi=e^{-u}$. If $d\phi\neq0$, then $\mathcal{W}_{abc}^{\ \ \ d}\phi_d=0$.
\end{proposition}
\begin{proof}
Commute derivatives on the Hessian BEC to get
\begin{align}
R_{abc}^{\ \ \ d}\phi_d=K(g_{ac}\phi_b-g_{bc}\phi_a).\label{reduced-0}
\end{align} 
Since $R_{ab}=\Lambda g_{ab}$, using the decomposition of the Riemann tensor into its Weyl and Ricci parts we get
\begin{align}
R_{abc}^{\ \ \ d}\phi_d=\mathcal{W}_{abc}^{\ \ \ d}\phi_d+\frac{\Lambda}{(n-1)}(g_{ac}\phi_b-g_{bc}\phi_a).
\end{align}
Comparing this to \eqref{reduced-0} gives the Weyl-alignment condition $\mathcal{W}_{abc}^{\ \ \ d}\phi_d=0$.\hfill\qed
\end{proof}

We now prove the following classification result.
\begin{theorem}\label{th-17}
Let $(M,g_{ab},u,Y_{ab},\lambda)$ be a Einstein CKY--RS structure, and suppose $d\phi\neq0$. Set $\alpha=c-\lambda$. Then, either 
\begin{enumerate}
\item $\alpha=0$, in which case $\mathcal{L}_{\xi}\phi=0$ and $\xi^a$ is a KV for both $g_{ab}$ and $\tilde g_{ab}$; or
\item $\alpha\neq0$, in which case $q=c=0$, $\mathcal{L}_{\xi}\tilde g_{ab}=-2\alpha \tilde g_{ab}$, and $\tilde g_{ab}$ is Ricci-flat. 
\end{enumerate}
\end{theorem}
\begin{proof}
Differentiating the eigenvalue/compatibility equation $\mathcal{L}_{\xi}\phi=\alpha\phi$  and using the BEC gives
\begin{align*}
(\nabla_a\xi_b)\phi^b=\alpha\phi_a+K\phi\xi_a.\label{funda-new-1}
\end{align*}
Contracting with $\phi^a$ and using $\mathcal{L}_{\xi}\phi=\alpha\phi$ we get $\alpha(|\nabla\phi|^2+K\phi^2)=0$. If $\alpha=0$, then $\mathcal{L}_{\xi}\phi=0$ and $\xi^a$ is also a KV for $\tilde g_{ab}$. If $\alpha\neq0$, then $|\nabla\phi|^2+K\phi^2=0$, and using and the first integral \eqref{first-int-1} gives $c=(n-1)q\phi$. Since $d\phi\neq0$, $c=q=0$. Direct computation gives $\mathcal{L}_{\xi}\tilde g_{ab}=g_{ab}\mathcal{L}_{\xi}\phi^{-2}=-(2\phi^{-3}\mathcal{L}_{\xi}\phi)g_{ab}=-2\alpha\tilde g_{ab}$.\hfill\qed
\end{proof}

The following corollary immediately follows:
\begin{corollary}\label{cor-x}
Under the assumptions of Theorem~\ref{th-17}, suppose that $\alpha\neq0$. Then
\begin{equation}\label{restr}
(n-1)|\nabla\phi|^2=-\Lambda\phi^2.
\end{equation}
Consequently, the signature of $g_{ab}$ imposes the following scalar curvature constraints:
\begin{enumerate}
\item If $g$ is Riemannian, then necessarily $\Lambda<0$. 
\item If $g$ is Lorentzian, then the causal character of $d\phi$ is determined by the sign of $\Lambda$:
\begin{align*}
\begin{cases}
\Lambda>0 &\Longrightarrow d\phi \text{ is timelike},\\
\Lambda=0 &\Longrightarrow d\phi \text{ is null},\\
\Lambda<0 &\Longrightarrow d\phi \text{ is spacelike}.
\end{cases}
\end{align*}
\end{enumerate}
\end{corollary}

\begin{theorem}[Compact case]\label{thm-compact}
Let $(M,g_{ab},u,Y_{ab},\lambda)$ be a compact Riemannian Einstein CKY--RS structure, and let $d\phi\neq0$. Then $\alpha=0$, $\Lambda>0$, $(M,g_{ab})$ is isometric to the round sphere of radius $r=\sqrt{(n-1)/\Lambda}$, and
\begin{align*}
\phi=\bar q+\bar\phi,\quad \bar\phi_{ab}=-r^{-2}\bar\phi g_{ab},\quad \bar q>\bar\phi_{\rm {max}}, \quad c=\Lambda(\bar q^2-\bar\phi^2_{\rm {max}}).
\end{align*}
Consequently, the conformal metric $\tilde g_{ab}$ is the round sphere.
\end{theorem}
\begin{proof}
Since $\xi^a$ is Killing, it is volume preserving and hence  
\begin{align*}
\int_M\xi(\phi)\,\mbox{dVol}_g=\alpha\int_M\phi\mbox{dVol}_g=0.
\end{align*}
Thus, $\alpha=0$. If $\Lambda=0$, $\Delta\phi=nb$. Integrating gives $q=0$, and therefore $\Delta\phi=0$, which gives $\phi$ is constant, contradicting $d\phi\neq0$. If $\Lambda\neq0$, define
\begin{align*}
\bar q=\frac{(n-1)}{\Lambda}q,\quad\bar\phi=\phi-\bar q.
\end{align*} 
This brings the BEC into the Obata form
 \begin{align*}
\bar\phi_{ab}=-\frac{\Lambda}{n-1}\bar\phi g_{ab}. 
\end{align*}
If $\Lambda<0$, then 
 \begin{align*}
\Delta\bar\phi=\left(\frac{n}{n-1}\right)|\Lambda|\bar\phi,
\end{align*}
and the maximum principle forces $\bar\phi\equiv0$, contradicting $d\phi\neq0$. If $\Lambda>0$, then by Obata's theorem \cite{Obata1} $(M,g_{ab})$ is isometric to the round sphere with radius $r=\sqrt{(n-1)/\Lambda}$. Since
\begin{align*}
\nabla_a(|\nabla\bar\phi|^2)=2\bar\phi^b\bar\phi_{ab}=-2r^{-2}\bar\phi\bar\phi_a=-r^{-2}\nabla_a\bar\phi^2,
\end{align*}
$|\nabla\bar\phi|^2+r^{-2}\bar\phi^2=x_0$, for some constant $x_0$. At the maximum of $\bar\phi_{\rm {max}}$ of $\bar\phi$, $x_0=r^{-2}\bar\phi^2_{\rm {max}}$ and therefore
\begin{align*}
|\nabla\bar\phi|^2=r^{-2}(\bar\phi^2_{\rm {max}}-\bar\phi^2).
\end{align*}
Comparing to \eqref{first-int-1} gives $c=\Lambda(\bar q^2-\bar\phi^2_{\rm {max}})$. Since $\phi>0$, $\bar q>\bar\phi^2_{\rm {max}}$ and so $c>0$. The conformal metric $\tilde g_{ab}$ is Einstein with constant positive curvature, and it is conformal to the round sphere with $d\phi\neq0$. Hence, $\tilde g_{ab}$ also the round sphere.\hfill\qed
\end{proof}

Finally, we state the following local classification result. 

\begin{theorem}\label{th-108}
Let $(M,g_{ab},u,Y_{ab},\lambda)$ be an Einstein CKY--RS structure, and suppose $d\phi\neq0$ on a neighborhood $U$ of a point $p\in M$. 
\begin{enumerate}
\item If $|\nabla\phi|^2\neq0$ at $p$, define $\epsilon=\operatorname{sgn}(|\nabla\phi|^2)\in\{\pm1\}$. Then there exists a coordinate $\tau$ with $|\nabla\tau|^2=\epsilon$ and $\phi=\phi(\tau),\dot{\phi}>0$ such that locally,
\begin{align*}
g&=\epsilon\,d\tau^2+\bar\gamma(\tau)^2h,\quad c_0\bar\gamma=\dot{\phi},\quad\ddot{\phi}=\epsilon\left(q-\frac{\Lambda}{n-1}\phi\right),\\
R^h_{AB}&=\mu h_{AB},\quad \mu=\frac{\epsilon}{c_0^{2}}(n-2)\left(q^2-\frac{\Lambda c}{(n-1)^2}\right),\quad \dot{\phi}=\frac{d\phi}{d\tau},
\end{align*}
with $c_0$ a non-zero constant.

\item $\Lambda=0$: for $q=0$, $\bar\gamma$ is constant and $g=\epsilon d\tau^2+h$ with $h_{AB}$ Ricci-flat; for $q\neq0$, $\bar\gamma=\tau$ after a shift of $\tau$ and $c_0=\epsilon q$, so $g=\epsilon d\tau^2+\tau^2h$ with $R^h_{AB}=\epsilon(n-2)h_{AB}$ (Ricci-flat cone).

\item $\Lambda\neq0$: set $k=\sqrt{|\Lambda|/(n-1)}$ and $\mathcal{D}:=\bar q^2-c/\Lambda$. Then, for $\dot{\phi}>0$,
\begin{center}
\begin{tabular}{|l|l|l|}
\hline
Background curvature & \qquad \qquad \qquad Conformal factor & Fiber curvature \\
\hline
$\epsilon\Lambda>0$ & $\quad\bar\phi=\mathcal{A}_1\cos k\tau+\mathcal{A}_2\sin k\tau=\sqrt{\mathcal{D}}\cos k(\tau-\tau_0)$, & \quad $\mathcal{D}=\mathcal{A}_1^2+\mathcal{A}_2^2$\\
\hline
$\epsilon\Lambda<0$ & $\quad\bar\phi=\mathcal{B}_1e^{k\tau}+\mathcal{B}_2e^{-k\tau}$, & \quad$\mathcal{D}=4\mathcal{B}_1\mathcal{B}_2$\\
\hline
& $\quad c=\Lambda(\bar q^2-\mathcal{D})$. & \\
\hline
\end{tabular}
\end{center}

\item If $\alpha=0$, then $\xi^a$ is tangent to the level sets of $\tau$ and a Killing field of $(N,h_{ab})$.

\item If $\alpha\neq0$ and $|\nabla\phi|^2\neq0$, then $q=c=0$, $\epsilon\Lambda<0$, $\phi=\phi_0e^{k\tau}$, $g=\epsilon d\tau^2+e^{2k\tau}h$ with $h_{ab}$ Ricci--flat, and $\tilde g=\phi_0^{-2}(\epsilon\,ds^2+h)$, where $s=-(1/k)e^{-k\tau}$. Moreover $\xi=(\alpha/k)\partial_\tau+\mathcal{N}$ with $\mathcal{N}$ being the component tangent to $N$, and $\mathcal{N}$ a proper homothety for $h_{AB}$: $\mathcal L_{\mathcal{N}}h_{ab}=-2\alpha h_{AB}$.

\item If $g_{ab}$ is Lorentzian and $|\nabla\phi|^2\equiv0$ on $U$, then $\Lambda=q=c=0$, $\nabla\phi$ is a parallel null vector field, and $g_{ab}$ is locally of the Brinkmann class.
\end{enumerate}
\end{theorem}
\begin{proof}
The metric splitting follows from standard arguments. (See the references \cite{Brinkmann2, Tashiro1, Kuhnel1}, adapted to the Riemannian setting. We provide a short proof here for completeness, covering the Lorentzian setting as well.)

\textit{1.} Let $\tau$ be a function and define $\sqrt{||\nabla\phi|^2|}d\tau=d\phi$. Then $|\nabla\tau|^2=\epsilon$ and $\phi=\phi(\tau)$ with $\dot{\phi}=\sqrt{||\nabla\phi|^2|}>0$, and hence $|\nabla\phi|^2=\epsilon\dot{\phi}^2$. Now, since $\nabla_a\tau$ has constant norm, it is geodesic. Since 
\begin{align*}
\phi_{ab}=\ddot{\phi}\,\tau_a\tau_b+\dot{\phi}\,\tau_{ab},
\end{align*}
using $\phi_{ab}=\psi g_{ab}$ with $\psi=-(\Lambda/(n-1))\phi+q$,  and contracting with the product $\tau^a\tau^b$ gives $\ddot{\phi}=\epsilon\psi$. Thus
\begin{align*}
\mbox{Hess}(\tau)=\frac{\psi}{\dot{\phi}}(g-\epsilon\,d\tau^2),
\end{align*}
from which it follows that the level sets of $\tau$ are totally umbilic with mean curvature $\psi/\dot{\phi}$ being $\tau$-dependent only. Therefore $g_{ab}$ splits as
\begin{align}
g=\epsilon d\tau^2+\bar\gamma(\tau)^2h,\label{split-met}
\end{align}
for some function $\bar\gamma(\tau)$. This metric \eqref{split-met} has $\Gamma^\tau_{AB}=-\epsilon \bar\gamma\dot{\bar\gamma}h_{AB}$, which gives $\phi_{AB}=-\dot{\phi}\Gamma^\tau_{AB}=\epsilon \bar\gamma\dot{\bar\gamma}\dot{\phi}h_{AB}$. From the BEC $\phi_{AB}=\psi g_{AB}=\psi \bar\gamma^2h_{AB}=\epsilon\ddot{\phi}\bar\gamma^2h_{AB}$, it follows $\dot{\bar\gamma}/\bar\gamma=\ddot{\phi}/\dot{\phi}$. Hence, $c_0\bar\gamma=\dot{\phi}$ for constant $c_0$.

The components of the Ricci tensor o metric of this form \eqref{split-met} are computed explicitly as \cite[Corollary 7.43]{ONeill1}
\begin{align*}
R_{\tau\tau}=-(n-1)\frac{\ddot{\bar\gamma}}{\bar\gamma},\quad R_{AB}=R^h_{AB}-\epsilon(\bar\gamma\ddot{\bar\gamma}+(n-2)\dot{\bar\gamma}^2)h_{AB}.
\end{align*}
Since $R_{ab}=\Lambda g_{ab}$, $R_{\tau\tau}=\epsilon\Lambda h_{AB},\quad R_{AB}=\Lambda\bar\gamma^2h_{AB}$, from which we obtain $(n-1)\ddot{\bar\gamma}=-\epsilon\Lambda \bar\gamma$ and 
\begin{align*}
R^h_{AB}&=(\Lambda \bar\gamma^2+\epsilon\,(\bar\gamma\ddot{\bar\gamma}+(n-2)\dot{\bar\gamma}^2))h_{AB}=(n-2)\left(\frac{\Lambda}{n-1}\bar\gamma^2+\frac{\epsilon}{c_0^2}\psi^2\right)h_{AB}\\
&=\frac{\epsilon}{c_0^2}(n-2)\left(\frac{\Lambda}{n-1}\frac{c_0^2\bar\gamma^2}{\epsilon}+\psi^2\right)h_{AB}=\frac{\epsilon}{c_0^2}(n-2)\left(\frac{\Lambda}{n-1}\frac{\dot{\phi}^2}{\epsilon}+\psi^2\right)h_{AB}\\
&=\frac{\epsilon}{c_0^2}(n-2)\left(\frac{\Lambda}{n-1}|\nabla\phi|^2+\psi^2\right)h_{AB}=\frac{\epsilon}{c_0^2}(n-2)\left(q^2-\frac{\Lambda}{(n-1)^2}c\right)h_{AB}.
\end{align*}
where in the last line we have also used the first integral \eqref{first-int-1}.

\vline

\textit{2.} For $\Lambda=0$, $\ddot{\phi}=\epsilon q$. If $q=0$, $\bar\gamma$ is constant and $R^h_{AB}=0$. If $q\neq0$, $\dot{\phi}=\epsilon q\tau$ (up to a constant shift). Choosing $c_0=\epsilon q$, we have $\bar\gamma=\tau$ and $R^h_{AB}=\epsilon(n-2)h_{AB}$.

\vline

\textit{3.} For $\Lambda\ne0$, we solve the barred $\phi$ equation $\ddot{\bar\phi}=-\epsilon\Lambda\bar\phi/(n-1)$. For $\epsilon\Lambda>0$ this is harmonic, for $\epsilon\Lambda<0$ hyperbolic as obtained. We substitute $|\nabla\phi|^2=\epsilon\dot{\phi}^2$ in the first integral \eqref{first-int-1}, and use the transformation $\phi\to\bar\phi$ to get
\begin{align*}
\Lambda\bar q^2-c=\Lambda\left(\bar\phi^2+\frac{\epsilon}{\Lambda}(n-1)\right).
\end{align*} 
The parenthesized term is conserved and we set this to $\mathcal{D}$. Then, $c=\Lambda(\bar q^2-\mathcal{D})$. For the pair of solutions for $\bar\phi$ we compute respectively i) harmonic: $\mathcal{D}=\mathcal{A}_1^2+\mathcal{A}_2^2$ and ii) hyperbolic: $\mathcal{D}=4\mathcal{B}_1\mathcal{B}_2$.

\vline

\textit{4.} If $\alpha=0$, $\mathcal L_{\xi}\phi=0$, which gives $\mathcal L_{\xi}\tau=0$. Hence, $\xi^a=\xi^A(\tau,x)\partial_A$. Since $\mathcal L_\xi h_{AB}=0$ for $(A,B)\neq(\tau,\tau)$, $\xi^A$ is $\tau$-independent, and $\xi^a$ is KV for $h_{AB}$.

\vline

\textit{5.} By Theorem \ref{th-17}, if $\alpha\neq0$, $q=c=0$ and $|\nabla\phi|^2=-\Lambda\phi^2/(n-1)$. Using $|\nabla\phi|^2=-\dot{\phi}^2$ gives
\begin{align*}
\dot{\phi}^2=-\epsilon\frac{\Lambda}{(n-1)}\phi^2,
\end{align*} 
and hence $\epsilon\Lambda<0$ and $\dot{\phi}=k\phi$. Therefore, $\phi=\phi_0e^{k\tau}$, $\bar\gamma=(k\phi_0/c_0)e^{k\tau}$, and $\mu=0$. Choose $c_0=k\phi_0$, and define $s=-(1/k)e^{-k\tau}$. Then $g=k^2s^2(\epsilon ds^2+h)$, and $\tilde g=\phi_0^{-2}(\epsilon ds^2+h)$. Now, from $\mathcal{L}_{\xi}\phi=\alpha\phi$, we have $\mathcal{L}_{\xi}\tau=\alpha/k$. Hence we have $\xi^a=(\alpha/k)\partial_\tau+\mathcal{N}$, where $\mathcal{N}$ is the component of $\xi^a$ tangent to $N$. The $A-B$ component of $\mathcal{L}_{\xi}g_{ab}$ decomposes as $0=\mathcal{L}_{\xi}g_{AB}=2\alpha e^{2k\tau}h_{AB}+e^{2k\tau}\mathcal L_{\mathcal{N}}h_{AB}$. 

\vline

\textit{6.} If $|\nabla\phi|^2\equiv0$ on $U$, this forces $b=\Lambda=c=0$, and therefore $\phi_{ab}=0$. Hence $\nabla\phi$ is parallel. Since it is a parallel non--zero null 1-form, and $g_{ab}$ is Lorentzian, $g_{ab}$ is locally of the Brinkmann class \cite{Brinkmann1}.\hfill\qed
\end{proof}


\section{Explicit CKY--RS constructions}\label{sec-5}


We show how certain Lorentzian geometries realize the CKY--RS structures. We focus on geometries with static spherical (circular) symmetry. We begin by establishing a general result for spherically symmetric CCKY 2--forms in spherically symmetric geometries.
\begin{proposition}\label{prop-a}
Let $M=M^2\times \mathbb{S}^2$ be the spherically symmetric product
\begin{align}\label{ss-product}
ds^2=h_{AB}(x)dx^Adx^B+r^2(x)d\Omega^2,\quad dr\neq0.
\end{align}
where $d\Omega^2$ is the canonical metric on the unit sphere and $\epsilon_{AB}$ is the volume form on $(M^2,h_{AB})$. Within the spherically symmetric class, every closed CKY 2--forms is, up to an overall normalization,
\begin{align}
Y=r\epsilon_{AB},\label{unique-ccky}
\end{align}
and its associated vector is unique up, and is precisely the Kodama vector field
\begin{align}
\xi^A=-\epsilon^{AB}\nabla_B r,\quad \xi^i=0.\label{kodoma}
\end{align}
\end{proposition}
\begin{proof}
A generic CKY 2--form on this product is
\begin{align}
Y=Q_1(x)\epsilon_h+Q_2(x)\omega_{\mathbb{S}^2}.\label{gen-sph-sym}
\end{align}
where $\omega_{\mathbb{S}^2}$ denotes the volume form on the unit sphere. Since $Y$ is closed $dQ_2=0$. From the mixed CKY equation $\nabla_AY_{ij}=2g_{A[i}\xi_{j]}$, $Q_2=0$ for $dr\neq0$. From the $M^2$--projected CKY equation $\nabla_AY_{BC}=2g_{A[B}\xi_{C]}$, we obtain $\nabla_A(r^{-1}Q_1)=0$, and hence $Q_1=Q_{1(0)}r$. Therefore, up to an overall normalization, we have \eqref{unique-ccky}. Consequently, the divergence of \eqref{unique-ccky} gives \eqref{kodoma}, which is precisely the Kodoma vector field.\hfill\qed
\end{proof}

The relation between the Kodama vector and the associated vector of a CCKY 2--form was recently established by Kinoshita in \cite{Kinoshita1}, who showed that the origin of the Kodama vector field can be understood in terms of an associated CKY 2--form and, in particular, that the standard warped-product spacetimes admitting a Kodama vector possess a CCKY 2--form. As a complementary result, we strengthen this observation within the spherically symmetric class of CCKY 2--forms by showing that the CCKY 2--form \eqref{unique-ccky} is unique up to normalization. Consequently, its associated vector is  precisely uniquely the Kodama vector.

Under a conformal transformation in general spherical symmetry, a non-trivial transformation is only realizable if closedness is not transported to the conformal metric. 
\begin{proposition}\label{xa}
Let $M$ be the spherically symmetric product \eqref{ss-product}. Let $Y_{ab}$ be a spherically symmetric CCKY 2--form and let $\tilde g=e^{2u}g$ with $u$ being spherically symmetric. If $Y_{ab}$ is non--degenerate, then $dr=0$. Moreover, if closedness of $Y_{ab}$ is preserved under the transformation, then $du=0$ and hence the conformal factor is constant.
\end{proposition}
\begin{proof}
Let \eqref{gen-sph-sym} be the CCKY 2--form. Since $Y_{ab}$ is non--degenerate, $Q_j\neq0$. From the mixed CKY equation $\nabla_AY_{ij}=2g_{A[i}\xi_{j]}$, since $Q_j\neq0$, $dr=0$. Now, since $du\wedge\epsilon_h=0$ (these are both forms on $M^2$), we have
\begin{align*}
du\wedge Y=Q_2du\wedge\omega_{\mathbb{S}^2}.
\end{align*}
Suppose $Y_{ab}$ is also closed. Then, the exterior derivative of the conformally related CKY 2--form is
\begin{align*}
d\tilde Y=e^{3u}(3du\wedge Y+dY)=3e^{3u}du\wedge Y.
\end{align*} 
If $\tilde Y_{ab}$ is also closed, then $du\wedge Y=0$. Since $Q_j\neq0$, $du=0$.\hfill\qed
\end{proof}

This result further demonstrates the rigidity of non--degeneracy and the transporting properties of the geometry of quantities of the CKY--RS structure. 

\begin{enumerate}
\item \textbf{Static spherically symmetric CKY-RS family.} We will now restrict 4--dimensional static spherically symmetric (SSS) product. Here, we choose the ansatz for the background, choose a suitable conformal representative, and then construct a CKY--RS structure. Consider the ansatz for a SSS metric:

\begin{align}\label{sectionone}
ds^{2} = -f dt^{2} + f^{-1} dr^{2} + r^{2}\left(d\theta^{2} + \sin^{2}\theta d\phi^{2}\right),
\end{align}
with $f = f\left(r\right)$ positive and arbitrary. This it the background metric $g_{ab}$ for the pending CKY--RS construction. By Proposition \ref{prop-a}, the unique spherically symmetric CCKY data is, up to a normalization,
\begin{align}
Y &= rdt \wedge dr, \quad \xi^{a}= \partial_{t}. \label{sectiontwo}
\end{align}
(Note $Y \wedge Y =0$, and hence $Y_{ab}$ is degenerate). Let $u = u(r)$. Since $\xi_a$ is a KV for the background metric, the conformal metric is Einstein, and $\lambda = c$ (since $\partial_tu=0$). Let the conformal representative $\tilde g_{ab}$ be the Schwarzschild--de Sitter geometry:
\begin{align}\label{sectionfive}
\tilde{d}s^{2} &= -\tilde{f} dt^{2} + \tilde{f}^{-1} d\tilde{r}^{2} + \tilde{r}^{2}\left(d\theta^{2} + \sin^{2}\theta d\phi^{2}\right),\quad \tilde{f}= 1 - \frac{2\mathcal{M}}{\tilde{r}} - \frac{c\tilde{r}^{2}}{3}.
\end{align}
By comparing the temporal and angular parts of (\ref{sectionone}) and (\ref{sectionfive}), we can combine the resulting relations to get 
\begin{align}\label{sectioneight}
u^{\prime} = r^{-1}\left(\pm e^{u}-1\right).
\end{align}
We choose the positive branch. Using $\phi = e^{-u}$, equation (\ref{sectioneight}) has the general solution
\begin{align}
\phi = 1 + c_{3}r,
\end{align}
where $c_{3}$ is a constant. Hence, 
\begin{align}\label{sectionnine}
f\left(r\right) = \left(1 + c_{3} r\right)^{2} - \frac{2\mathcal{M}}{r}\left(1 + c_{3}r\right)^{3} - \frac{\lambda}{3}r^{2}.
\end{align}
The full CKY--RS data  on \eqref{sectionfive} is therefore
\begin{align*}
\left(M, g_{ab}^{c_3}, -\ln\left(1 + c_{3}r\right), rdt \wedge dr, \partial_t,\lambda\right);
\end{align*}
where $c_{3}$ is a constant, $g_{ab}$ is given by (\ref{sectionone}) and $f\left(r\right)$ by (\ref{sectionnine}). 

\item \textbf{3-dimensional BTZ CKY--RS constructions.}\\

\textit{A 3-dimensional BTZ CKY--RS family --} The three-dimensional construction is identical in structure to the 4--dimensional one. Starting from the BTZ Einstein representative \cite{Banados1}
\begin{align}
\tilde{ds}^{2}& = -\tilde{f} dt^{2} + \tilde{f}^{-1} dr^{2} + r^{2} d\varphi^{2},\quad\tilde{f}= \frac{r^{2}}{l^{2}} - \mathcal{M}=\frac{r^{2} - r_{h}^{2}}{l^{2}},\quad r_{h}^{2}= l^{2}\mathcal{M},\label{btz-1} 
\end{align}
($\mathcal{M} \neq 0$) and using $(Y = rdt \wedge dr, \xi^{a} = \partial_{t})$ the CKY--RS compatibility gives the exact solution as in the 4-dimensional: $\phi=1+c_3r$, which selects $f$:
\begin{align}
f=\frac{r^2}{l^2}-\mathcal{M}(1+c_3r)^2.\label{new-btz}
\end{align}
Hence, the full CKY--RS data  on \eqref{sectionfive} is
\begin{align*}
    \left(M, g_{ab}^{c_3}, -\ln\left(1 + c_{3}r\right), rdt \wedge dr, \partial_t,-2/l^2\right).
\end{align*}

\vline

\textit{A BTZ realization of the Ricci--flat conformal branch --} The static $c_3$--family above is a CKY--RS construction obtained by conformally deforming the Einstein BTZ metric; it therefore does not constitute a direct realization of the Einstein background analysis of Section \ref{sec-4}. In the present construction, the CKY--RS structure is constructed directly on the BTZ metric, while the conformally related representative is flat and trivially Einstein. This example therefore provides anexplicit realization of the conformal Einstein setting analyzed in Section \ref{sec-4}, including its non--constant BEC and $K\neq0$ branch.

Consider the $3$--dimensional (non--rotating) BTZ metric \eqref{btz-1}, with $\mathcal{M} \neq 0$. This metric is Einstein with $R_{ab} = -\left(2/l^{2}\right)g_{ab}$. This metric is in the $\phi$ non--constant and $K\neq0$ branch. Hence, $K = -1/l^{2}$.  The first integral \eqref{first-int-1} is therefore $|\nabla \phi |^{2} = (l^{-1}\phi)^{2} + C$. Combining this with the expression for the conformal Ricci curvature and the Hessian BEC gives $\tilde{R}_{ab} = -2C\tilde{g}_{ab}$, where $C$ is a constant. The Einstein constant of the conformal metric is related to the compatibility constant $c$ by $c = -2 C$.

The CCKY data, as before, is 
\begin{align}
Y &= rdt \wedge dr,\quad \xi^{a}= \partial_{t}.
\end{align}
Now, the soliton equation gives $\mathcal{L}_{\xi}\tilde{g}_{ab} = 2\left(\lambda + 2C\right)\tilde{g}_{ab}$. Since $\mathcal{L}_{\xi}\tilde{g}_{ab} = \mathcal{L}_{\partial_{t}}\tilde{g}_{ab}= -2\left(\phi^{-1}\phi_{t}\right)\tilde{g}_{ab}$, we have
\begin{align}\label{sectionbseven}
\frac{\phi_{t}}{\phi} = -\left(\lambda + 2C\right),
\end{align}
which has solution $\phi = \zeta\left(r, \varphi\right)e^{\bar{w }t}$, with $\bar w = -(\lambda + 2C)$. By the Hessian BEC, $\bar w \zeta_{\phi} =0$, and
\begin{align}
\left(\bar f\bar f^{\prime}\right)\zeta^{\prime} - 2\left(\bar w^{2} + \frac{\bar f}{l^{2}}\right)\zeta &= 0,\label{sectionbeleven}\\
\bar f\zeta^{\prime} - \frac{r}{l^{2}}\zeta &= 0.\label{sectionbtwelve}
\end{align}
Comparing \eqref{sectionbeleven} and \eqref{sectionbtwelve} we get $\bar w = \pm l^{-2}r_{h}= \pm l^{-1}\sqrt{\mathcal{M}}$. Since $\mathcal{M}\neq0$ ($\mathcal{M}=0$ forces the $K=0$ branch), $\bar w\neq0$, and therefore $\zeta=\zeta(r)$. One explicitly integrates \eqref{sectionbtwelve} to obtain the time--dependent conformal factor
\begin{align}
\phi=\phi_0\sqrt{r^2-r_h^2}e^{\bar wt}.\label{td-phi}
\end{align}

The norm of the gradient of \eqref{td-phi} is $|\nabla\phi|^2=(l^{-1}\phi)^2$. Compare to the first integral \eqref{sectionbfour} and it follows that $C=0$. Hence, $\tilde R_{ab}=0$. Since the manifold is 3--dimensional, $\tilde R_{abcd}=0$ and therefore $\tilde g_{ab}$ is flat. The soliton constant is $\lambda=-\bar w=\pm l^{-1}\sqrt{\mathcal{M}}$. The full CKY--RS data on the 3--dimensional (non-rotating) BTZ, for $\mathcal{M}>0$ and $r>r_h$, is therefore given by
\begin{align*}
(M,g_{ab},u_0\pm l^{-1}\sqrt{\mathcal{M}}t-\ln\sqrt{r^2-r_h^2},rdt\wedge dr,\partial_t,\mp l^{-1}\sqrt{\mathcal{M}}).
\end{align*}

We emphasize that $\phi\rightarrow0$ at $r=r_h$, and therefore becomes a conformal boundary of the transformed metric for $\tilde g_{ab}$. That is, $u$ diverges there, and therefore the conformal factor is singular at the horizon. This is why we carefully defined the CKY--RS structure away from the horizon. 

Since the conformal metric is locally flat, it is then natural to express it in Minkowski coordinates. The coordinate transformation
\begin{align*}
T=\frac{l}{\phi}\frac{r}{r_h}\cosh\left(\frac{r_h}{l}\varphi\right),\quad X=\frac{l}{\phi}\frac{r}{r_h}\sinh\left(\frac{r_h}{l}\varphi\right),\quad Z=\frac{l}{\phi},
\end{align*}
brings the time--dependent conformal metric $\tilde g_{ab}$ into the locally flat form
\begin{align*}
ds^2=-dT^2+dX^2+dZ^2.
\end{align*} 

\end{enumerate}


\section{Discussion and outlook}\label{sec-6}


We have introduced conformal Killing-Yano Ricci soliton (CKY--RS) geometries as a class of conformal Ricci solitons in which the soliton vector field is fixed by a rank--2 CKY tensor through its divergence. The CKY--RS structure condition couples CRS to the hidden symmetry structure of the background geometry and leads to non-trivial compatibility conditions involving the Weyl, Cotton, and CKY tensors. The resulting equations are therefore naturally overdetermined. 

Another important consequence of this new structure is the rigidity of the four-dimensional Lorentzian metrics with a principal CKY 2-form. The classification of \cite{Krtous1} forces the conformal metric into the Kerr--NUT--(A)dS family of metrics, provided the principal structure is preserved. This, however, forces a constant coonformal factor. That is, a principal CKY--RS structure is rigid against non-trivial conformal deformations.

The analysis on Einstein backgrounds simplifies and gives a more transparent picture. If the background is Einstein, the CKY--associated vector field becomes Killing vector, and the conformal factor satisfies both an eigenvalue equation along the CKY--flow and an Obata-type Hessian equation, and the soliton (i.e. conformal) metric is necessarily Einstein, with $\tilde R=c \tilde g_{ab}$ and the soliton constant given by$\lambda=c-\alpha$. A dichotomy is provided by the constant $\alpha$: it's vanishing imposed that the soliton vector is a KV for both the background and conformal metrics, and if $\alpha\neq0$ this forces the conformal metric to be Ricci-flat. The Hessian equation forces the background into a restricted class of geometries. Locally, the background is a warped product: $g=\epsilon d\tau^2+\bar\gamma(\tau)^2h$ with $\bar\gamma\propto\dot{\phi}=d\phi/d\tau$ and $h_{AB}$ being the metric on the fiber $N$. These metrics include products, Ricci--flat cones, and metrics of the Brinkmann class. For $\alpha=0$ the soliton (CKY) vector field $\xi^a$ is a KV on the fiber, so the CKY--RS data reduces to a KV on $N$. If $\alpha\neq0$, $\xi^a$ is a proper homothety. In the compact Riemannian case, $\alpha$ necessarily vanishes and $(M,g_{ab})$ is isomorphic to the round sphere.

Finally we construct simple but explicit examples of CKY--RS structures on Lorentzian geometries, specialized to the spherically symmetric setting. In this setting, the associated CKY vector is precisely the Kodama vector, providing a geometric interpretation of the preferred soliton vector field. The Kodama vector field is defined without requiring a stationary Killing symmetry and plays an important role in the analysis of dynamical spherically symmetric spacetimes. Hence, its appearance here, in addition to the recent identification of the Kodama vector as arising from a CCKY 2--form strengthens the interpretation of the present CKY--RS construction as a hidden symmetry formulation of geometrically preferred flows in general relativity \cite{Kinoshita1}. A natural follow--up investigation would be to investigate whether the equations of the CKY--RS structure yield useful quasi--local or conserved quantities associated with the Kodama flow, and whether these quantities admit interpretations in terms of the field equations of general relativity. 

Explicit CKY--RS structures were constructed in 4 and 3--dimensional static spherical (circular) symmetry, producing a parametrized family of metrics associated to a CKY--RS structure. A CKY--RS structure was then constructed on the 3--dimensional (non--rotating) BTZ metric, identifying a time--dependent flat conformal representative for the structure. This particular BTZ example illuminates something very crucial about the CKY--RS structure: the conformal representative selected by a CKY--RS structure need not preserve the curvature scale. In this case the time-dependent representative is flat, with the conformal factor becomes singular at the BTZ horizon. This exposes a global issue: the CKY--RS construction is fundamentally local, and a conformal factor admissible on some open region may fail to extend smoothly across a horizon or some special surface. Consequently, a complete theory of CKY--RS structures must distinguish local and global classification results. Questions concerning conformal boundaries, horizon regularity, completeness, and the global behaviour of the CKY flow are therefore natural extensions of the present work. 

Several other open questions remain. First, the full coupled integrability system is overdetermined, and it would worthwhile analyzing whether additional curvature restrictions or classificationresults follow from higher--order integrability conditions. Second, it would be natural to extend the rigidity analysis to higher dimensions, where principal CKY tensors provide a canonical framework for Kerr--NUT--(A)dS geometries, and to determine whether analogous obstructions to non-trivial CKY--RS structures persist. A systematic classification of CKY--RS structures on backgrounds besides Einstein metrics, without symmetry assumptions, is also worth pursuing. 

Another interesting direction is to determine whether CKY--RS structures arise naturally within known families of exact solutions of general relativity beyond the static spherical and BTZ examples constructed here, including rotating spacetimes and higher-dimensional black holes. Such developments may clarify whether the CKY--RS framework is primarily a classification tool for conformal geometric structures or whether it also selects physically distinguished families of exact spacetime metrics.

The results in this work therefore establish the basic structure and several rigidity and classification results of the CKY--RS structure while leaving the general classification, higher-dimensional rigidity, and global existence problems open to future pursuits.


\section*{Acknowledgement}


M. H acknowledges the support of the High-level Talent Research Start-up Project Funding of Henan Academy of Sciences (Project No.: 241819245). A. S research is supported by the Institute of Mathematics, funded through the High-level Talent Research Start-up Project Funding of the Henan Academy of Sciences (Project No.: 251819085). Y.-K. L is supported by Xiamen University Malaysia Research Fund (Grant no.: XMUMRF/2021-C18/IPHY/0011).


\begin{thebibliography}{99}


\bibitem{Hamilton1} 
R. S. Hamilton,
\textit{Three-manifolds with positive Ricci curvature},
J. Differential Geom., \textbf{17}, 255 (1982).

\bibitem{Hamilton2} 
R. S. Hamilton,
\textit{The Ricci flow on surfaces},
in Mathematics and General Relativity, Contemporary Mathematics \textbf{71}, edited by J. A. Isenberg (American Mathematical Society, Providence, RI, 1988), pp. 237–262.

\bibitem{Perelman1} 
G. Perelman,
\textit{The entropy formula for the Ricci flow and its geometric applications},
arXiv:math/0211159v1[math.DG] (2002).

\bibitem{Perelman2} 
G. Perelman,
\textit{Ricci flow with surgery on three manifolds},
arXiv:math/0303109v1[math.DG] (2003).

\bibitem{Brinkmann1} 
H. W. Brinkmann,
\textit{Riemann spaces conformal to Einstein spaces},
Math. Ann., \textbf{91}, 269 (1924).

\bibitem{Cao1} 
H.-D. Cao and Q. Chen,
\textit{On locally conformally flat gradient steady Ricci solitons},
Trans. Amer. Math. Soc., \textbf{364}, 2377 (2012).

\bibitem{Cao2} 
H.-D. Cao and Q. Chen,
\textit{On Bach-flat gradient shrinking Ricci solitons},
Duke Math. J., \textbf{162}, 1149 (2013).

\bibitem{Gover1} 
A. R. Gover and P. Nurowski,
\textit{Obstructions to conformally Einstein metrics in $n$ dimensions},
J. Geom. Phys., \textbf{56}, 450 (2006).

\bibitem{Catino1} 
G. Catino, P. Mastrolia, D. D. Monticelli, and M. Rigoli,
\textit{Conformal Ricci solitons and related integrability conditions},
Adv. Geom., \textbf{16}, 301 (2016).

\bibitem{Kashiwada1} 
T. Kashiwada,
\textit{On conformal Killing tensor},
Natur. Sci. Rep. Ochanomizu Univ., \textbf{19}, 67 (1968).

\bibitem{Kashiwada2} 
T. Kashiwada and S. Tachibana,
\textit{On the integrability of Killing-Yano’s equation},
J. Math. Soc. Japan, \textbf{21}, 259 (1969).

\bibitem{Yano1} 
K. Yano,
\textit{Some remarks on tensor fields and curvature},
Ann. Math., \textbf{55}, 328 (1952).

\bibitem{Walker1} 
M. Walker and R. Penrose,
\textit{On quadratic first integrals of the geodesic equations for type (\{22\}) spacetimes},
Commun. Math. Phys., \textbf{18}, 265 (1970).

\bibitem{Frolov1} 
V. P. Frolov and D. Kubiz\u{n}\'{a}k,
\textit{Hidden symmetries of higher-dimensional rotating black holes},
Phys. Rev. Lett., \textbf{98}, 011101 (2007).

\bibitem{Frolov2} 
V. P. Frolov and D. Kubiz\u{n}\'{a}k,
\textit{Higher-dimensional black holes: Hidden symmetries and separation of variables},
Class. Quantum Grav., \textbf{25}, 154005 (2008).

\bibitem{Krtous1} 
P. Krtou\u{s}, V. P. Frolov, and D. Kubiz\u{n}\'{a}k,
\textit{Hidden symmetries of higher-dimensional black holes and uniqueness of the Kerr-NUT-(A)dS spacetime},
Phys. Rev. D, \textbf{98}, 064022 (2008).

\bibitem{Kinoshita1} 
S. Kinoshita,
\textit{Geometrical origin of the Kodama vector},
Phys. Rev. D, \textbf{110}, 044056 (2024).

\bibitem{Obata1} 
M. Obata,
\textit{Certain conditions for a Riemannian manifold to be isometric with a sphere},
J. Differential Geom., \textbf{114}, 333 (1962).

\bibitem{Brinkmann2} 
H. W. Brinkmann,
\textit{Einstein spaces which are mapped conformally on each other},
Math. Ann., \textbf{94}, 119 (1925).

\bibitem{Tashiro1} 
Y. Tashiro,
\textit{Complete Riemannian manifolds and some vector fields},
Trans. Amer. Math. Soc., \textbf{117}, 251 (1965).

\bibitem{Kuhnel1} 
W. K\"{u}hnel,
\textit{Conformal Geometry},
(Editors. Ravi S. Kulkarni and Ulrich Pinkall, Volume E12, Vieweg, Braunschweig, 1988).

\bibitem{ONeill1} 
B. O'Neill,
\textit{Semi-Riemannian Geometry. With Applications to Relativity},
(Academic Press, New York, 1983).

\bibitem{Banados1} 
M. Ba\~{n}ados, C. Teitelboim, and J. Zanelli,
\textit{The black hole in three-dimensional spacetime},
Phys. Rev. Lett., \textbf{69}, 1849 (1992).

\end{thebibliography}
\end{document}